\documentclass[a4paper,12pt]{article}
\usepackage{latexsym}
\usepackage{amsmath}
\usepackage{amssymb}
\usepackage{enumitem}
\usepackage{tikz}
\usepackage{color}

\usepackage{theorem}
\newtheorem{theorem}{Theorem}
\newtheorem{proposition}{Proposition}
\newtheorem{lemma}{Lemma}

\newtheorem{problem}{Problem}

\newtheorem{MTA}{Theorem A}

\theorembodyfont{\rmfamily}
\newtheorem{proof}{\textmd{\textit{Proof.}}}

\newtheorem{remark}{Remark}
\newtheorem{example}{Example}

\newcommand{\qedd}{\hfill \Box}

\newcommand{\R}{\ensuremath{\mathbb{R}}}

\newcommand{\Sph}{\ensuremath{\mathbb{S}}}

\title{A new family of latitudinally corrugated two-spheres of revolution with simple cut locus structure
\footnote{
Mathematics Subject Classification (2010)\,:\,53C22.}
}
\author{ Minoru Tanaka, Toyohiro Akamatsu, \\ Robert Sinclair,  and Masaru Yamaguchi }
\date{}
\begin{document}


\maketitle

\begin{abstract}
There are not so many kinds of surface of revolution whose cut locus structure have been determined, although the cut locus structures of very familiar surfaces of revolution (in Euclidean space) such as ellipsoids, 2-sheeted hyperboloids, paraboloids and tori are now known.

Except for tori, the known cut locus structures are very simple, i.e., a single point or an arc.
In this article, a new family $\{M_n\}_n$ of 2-spheres of revolution with simple cut locus structure 
is introduced. This family is also new in the sense that the number of points on each meridian which assume a local minimum or maximum of the Gaussian curvature function
on the meridian goes to infinity as $n$ tends to infinity.

Thus, our family includes surfaces which have arbitrarily many bands of alternately
increasing or decreasing Gaussian curvature, although each member of this family has a simple cut locus structure.
\end{abstract}

\section{Introduction}

Determining the structure of the cut locus for a Riemannian manifold
is very difficult, even for a surface of revolution.
In 1994, Hebda \cite{H}  proved that the cut locus of a point $p$ in a complete 2-dimensional Riemannian manifold has a local tree structure and that the cut locus
is locally the image of a Lipschitz map from a subarc of the unit circle  in the unit tangent plane at $p$
(see also \cite{IT} and \cite[Theorem 4.2.1]{SST}).

In general, one cannot determine the structure of the cut locus in more detail for even a 2-dimensional Riemannian manifold.
 In fact, Gluck and Singer \cite{GS}
constructed a 2-sphere of revolution with positive Gaussian curvature which admits
a non-triangulable cut locus.
Fortunately, the cut locus structure has been determined for very familiar surfaces in Euclidean space, such as
paraboloids, hyperboloids, ellipsoids and tori of revolution (see \cite{E,GMST, IK1, IK2, ST2,T}).


These familiar surfaces, which are of revolution, are very useful as model surfaces, for obtaining global structure theorems of
Riemannian manifolds, if the cut locus structure is simple.
For example, special 2-spheres of revolution have been employed as model surfaces to obtain various sphere theorems
(see \cite{B, IMS, K}, for example).

Let $\gamma :[0,t_0]\to M$ denote a minimal geodesic segment on a complete Riemannian manifold $M.$ The endpoint $\gamma(t_0)$ is called a {\it cut point} of $p:=\gamma(0)$ along $\gamma$ if any geodesic extension of $\gamma$ (beyond $\gamma(t_0)$) is not minimal anymore. The {\it cut locus} of the point $p$ is defined as the set of cut points along all minimal geodesic segments emanating from $p.$

Let us introduce a new Riemannian metric $g$ on the unit 2-sphere $\Sph^2.$
Choose any smooth function $m:[0,\pi]\to [0,\infty)$ satisfying
\begin{equation}\label{1.1}
m(0)=0,\quad m'(0)=1,
\end{equation}
and
\begin{equation}\label{1.2}
m(\pi-r)=m(r)>0 
\end{equation}
for all $r\in(0,\pi).$
Then, one can introduce a Riemannian metric 
\begin{equation}\label{1.3}
g=dr^2+m(r)^2d\theta^2
\end{equation}
on the open subset  ${\Sph}^2\setminus\{N,S\}$ of ${\Sph}^2$, the poles removed from the sphere, where $(r,\theta)$ denotes geodesic polar coordinates around the south pole $S$, and $N$ denotes the north pole.

If $m$ is extendable to a smooth odd function around $r=0,$ then
it was proven in \cite[Lemma 7.1.1]{SST} that the Riemannian metric $g$ is extendable to a smooth one on the entirety of $\Sph^2.$

In this article, the Riemannian manifold $(\Sph^2,g)$ with (smooth) Riemannian metric $g$ in \eqref{1.3} is called a {\it 2-sphere of revolution} (with reflective symmetry).
By \eqref{1.2}, $(\Sph^2,g)$  is reflectively symmetric with respect to the {\it equator} $r=\pi/2.$
It is easy to check that the cut locus of the south pole $S$ of the 2-sphere of revolution is $\{N\},$ and vice versa (see \cite[Lemma 2.1]{ST2} for example).

The unit sphere is a typical example of a 2-sphere of revolution, whose metric is 
defined by $dr^2+\sin^2r d\theta^2.$
Ellipsoids of revolution are also 2-spheres of revolution with reflective symmetry.
They have the following remarkable property:

\noindent
{\it The Gaussian curvature is monotone from the south pole to the point on the equator
along each meridian.}

A {\it meridian} is the periodic geodesic $\{\theta=\theta_0\}\cup\{\theta=\pi+\theta_0\}\cup\{N,S\}$
for some constant $\theta_0\in(0,\pi].$
A {\it parallel} is the curve $r=r_0$ for some constant $r_0\in(0,\pi).$

It was proven in \cite{ST2} that the cut locus is a single point or an arc for 2-spheres of revolution with the same monotonicity property of the Gaussian curvature as ellipsoids.

Moreover, detailed numerical case studies described in \cite{ST1} led to a conjecture that the monotonicity properties of the
Gaussian curvature strongly control the cut locus structure for various kinds of surfaces of revolution.
Thus, 
due to the example by Gluck and Singer and the theorem and conjecture mentioned above,
the monotonicity properties of the Gaussian curvature may seem to be a suitable and reasonable sufficient condition for a 2-sphere of revolution to admit a simple cut locus structure.
However, a family of 2-spheres  $\{M_\lambda\}_\lambda:=\{(\Sph^2, dr^2+m_\lambda(r)^2d\theta^2)\}_{\lambda}$
of revolution with simple cut locus structure, but violating the monotonicity property, was introduced  in \cite{BCST}.
The Riemannian metric $dr^2+m_\lambda^2(r)d\theta^2$ of $M_\lambda$
is defined by
\begin{equation}\label{1.4}
m_\lambda(r):=\frac{\sqrt{\lambda+1}\sin r}{\sqrt{1+\lambda\cos^2r}}
\end{equation}
for each real number $\lambda\geq 0.$
It was proven in \cite{BCST} that for each $\lambda\geq0,$ 
the cut locus of a point $p\in r^{-1}(0,\pi)\subset M_\lambda$ is a point or subarc of the antipodal parallel
$r=\pi-r(p), $ and it was shown in \cite[Lemma 4.1]{BCST} that the Gaussian curvature of $M_\lambda$ assumes a unique minimum
on the parallel $r=\arccos\sqrt{2/\lambda}$ between the south pole and the equator along each meridian
if $\lambda>2.$
Incidentally, the Gaussian curvature of $M_\lambda$ is increasing along each meridian from the pole $r=0$ to the point on the equator $r=\pi/2$ if $\lambda\in[0,2].$

In this article, we introduce a new family of 2-spheres  $\{M_n\}_n$ of revolution
with the same cut locus structure as that of $\{M_\lambda\}_\lambda,$
but the number of points on each meridian which assume a local minimum or maximum
of the Gaussian curvature function along the meridian
goes to infinity as $n$ tends to infinity.

For each integer $n\geq2,$
we define a smooth function
$m_n:[0,\pi]\to [0,\infty)$ 
by
\begin{equation}\label{1.5}
m_n(r):=3\sin\left( r-\frac{\sin2r}{3}+\frac{\sin^22r\cdot\sin2n^2r}{n^5}\right).
\end{equation} 
Since
it is clear that the function $m_n(r)$ is extendable as a smooth odd function on $\R,$
we have a family of 2-spheres $\{M_n\}_n $ of revolution with Riemannian 
metric
$dr^2+m_n(r)^2d\theta^2.$

In this article, we will prove 
\begin{MTA}
Let $\{M_n\}_{ n }$ denote a family of 2-spheres of revolution with Riemannian metric 
 $dr^2+m_n(r)^2d\theta^2,$  where  
$m_n(r)$ denotes  the function defined by the equation \eqref{1.5}.
Then, there exists a number $n_0$ such that the cut locus of each point 
$p\in r^{-1}(0,\pi)$ is a subarc of the antipodal parallel $r=\pi-r(p)$ for all $n>n_0.$
Furthermore,
the number of points on each meridian which assume a local maximum or minimum of the Gaussian curvature function on the meridian goes to infinity as $n$ tends to infinity.
\end{MTA}

The family in {\bf Theorem A} confirms that the monotonicity of the Gaussian curvature is not a suitable or reasonable sufficient condition for a 2-sphere of revolution to admit a simple cut locus structure.

In Section 2, the first claim of {\bf Theorem A}  will be proven and the latter claim in Section 3.




\section{Preliminaries}

Let $f:[0,\pi/2]\to [0,\infty)$ denote a smooth function such that $f'(x)>0$ on the open interval $(0,\pi/2),$ and $f(0)=f'(\pi/2)=0.$
By making use of $f$,
the {\it half period function} $\varphi_f$ on $(0,f(\pi/2))$ is defined by
\begin{equation}\label{2.1}
\varphi_f(\nu):=2\int_{f^{-1}(\nu)}^{\pi/2} \frac{\nu}{f(x)\sqrt{f(x)^2-\nu^2}  } \, dx.
\end{equation}

\begin{lemma}\label{lem2.1}
If the function
$A(x):=\sqrt{a^2-f(x)^2}/f'(x),$  where $a:=f(\pi/2)$, is decreasing (respectively increasing)  on $(0,\pi/2)$, then
the function $\varphi_f(\nu)$ 
is decreasing (respectively increasing) on $(0,a).$
Furthermore, if the function $A(x)$ is strictly decreasing (respectively strictly increasing),
then the function $\varphi_f(\nu)$ is also strictly decreasing (respectively strictly increasing).
\end{lemma}
\begin{proof}
Putting $f(x)^2=u,$
we obtain,
$$\frac{dx}{f(x)\sqrt{f(x)^2-\nu^2}}=\frac{\sqrt{a^2-u}\;du}{2f'(x)\cdot u\sqrt{(u-\nu^2)(a^2-u)}}.$$
Hence, we get,

\begin{equation}\label{2.2}
\varphi_f(\nu)=\nu\int_{\nu^2}^{a^2} A\circ f^{-1}(\sqrt u)\frac{du}{u\sqrt{(u-\nu^2)(a^2-u)}  }.
\end{equation}
Next, setting
$s=\sqrt{\frac{u-\nu^2}{a^2-u} },$
we get
$$u=\frac{a^2s^2+\nu^2}{s^2+1},\quad \frac{du}{ds}=2s\frac{a^2-\nu^2}{(s^2+1)^2},
\quad u\sqrt{(u-\nu^2)(a^2-u)}=s\frac{(a^2-\nu^2)(a^2s^2+\nu^2)}{(s^2+1)^2}.$$
Thus,

\begin{equation}\label{2.3}
\varphi_f(\nu)=\nu\int_{0}^{\infty} A\circ f^{-1}(\sqrt u)\frac{2ds}{a^2s^2+\nu^2 }.
\end{equation}
By setting $s=\nu \tau,$
we have,
\begin{equation}\label{2.4}
\varphi_f(\nu)=2\int_{0}^{\infty} A\circ f^{-1}(\sqrt {u(\tau,\nu)})\frac{d\tau}{a^2\tau^2+1 },
\end{equation}
where
$u(\tau,\nu):=\frac{\nu^2(a^2\tau^2+1)}{\tau^2\nu^2+1}.$ 

Since $u(\tau,\nu)$ is increasing with $\nu$ for each $\tau\in(0,\infty),$ and since 
$A\circ f^{-1}$ is decreasing   (respectively increasing) on $(0,a),$
the function $A\circ f^{-1}(\sqrt { u(\tau,\nu)} )$ is decreasing  (respectively increasing) with $\nu  $ 
for each $\tau\in(0,\infty),$ and  $\varphi_f(\nu)$ is decreasing  (respectively increasing) on $(0, a).$ The latter claim is also obvious from this argument.

$\qedd$
\end{proof}

\begin{lemma}\label{lem2.2}
If the function
$\frac{-f''}{f}$   is increasing  (respectively decreasing) on $(0,\pi/2)$, then
the function $\varphi_f(\nu)$ 
is decreasing  (respectively increasing) on $(0,a).$
\end{lemma}
\begin{proof}
From Lemma \ref{lem2.1}
it is sufficient  to  prove  that the function
$A(x)$ is decreasing  (respectively increasing) on $(0,\pi/2).$
It is easy to check that
\begin{equation}\label{2.5}
A'(x)=\frac{-f(x)} {f'^2(x)\sqrt{a^2-f^2(x)} }F(x),
\end{equation}
where $F(x):= (f')^2(x)+\frac{f''}{f}(a^2-f^2)(x)$
and 
$$F'(x)=(a^2-f^2)\left\{{f''}/{f} \right\}'(x).$$
Hence, $F'(x)\leq0$ (respectively $F'(x)\geq0$).
Since the function $F$ is decreasing (respectively increasing) on $(0,\pi/2)$
and  $F(\pi/2)=0,$
 the function $F$ is positive (respectively negative)   and,  by \eqref{2.5},
$A'(x)\leq0$ (respectively $A'(x)\geq0)$ on $(0,\pi/2).$
$\qedd$
\end{proof}

By making use of Lemma \ref{lem2.1}, one can prove that the half period function $\varphi_\lambda$  for $m_\lambda$ is strictly decreasing without an explicit computation of $\varphi_\lambda$ (see \cite [Proposition 4.3]{BCST}).
\begin{lemma}\label{lem2.3}
The half period function $\varphi_\lambda(\nu)$ for the metric $m_\lambda$ is
strictly
decreasing on $(0,m_{\lambda}(\pi/2))$ for each $\lambda>0.$
\end{lemma}

\begin{proof}
It is easy to check that 
$$\sqrt{a^2-m_\lambda(r)^2}=\frac{(\lambda+1)\cos r}{\Lambda},$$ 
and
$$m_\lambda'(r)={\left(\sqrt{\lambda+1} /\Lambda \right)^3     \cos r},$$
where
 $\Lambda=\sqrt{1+\lambda\cos^2 r},$ and  $a=m_\lambda(\pi/2)=\sqrt{\lambda+1}.$
Therefore, the function 
$\sqrt{a^2-m_\lambda^2}/m_\lambda'={(1+\lambda\cos^2 r)}/{\sqrt{\lambda+1}}$
is strictly decreasing. By Lemma \ref{lem2.1},
the half period function $\varphi_\lambda$ is strictly decreasing on $(0,m_\lambda(\pi/2)).$
$\qedd$
\end{proof}

\begin{remark}
From  Lemma \ref{lem2.3} and \cite [Lemma 3.3, Theorem 3.5]{BCST}, it follows that 
the cut locus of a point $q\in r^{-1}(0,\pi)$ of the 2-sphere of revolution $(\Sph^2,dr^2+m_\lambda(r)^2d\theta^2)$ is a subarc of the antipodal parallel of the point $q.$
\end{remark}

Let $h :[0,\pi]\to R$ denote a smooth function which is
extendable to a smooth odd one
such that
\begin{equation}\label{2.6}
h'(x)>0 {\quad \text on \: \:}[0,\pi/2),
\end{equation}
\begin{equation}\label{2.7}
h(\pi-x)=\pi-h(x) {\quad \text on\: \:} [0,\pi].
\end{equation}
By substituting $x=\pi/2$ in the equation \eqref{2.7},
we get
\begin{equation}\label{2.8}
h(\pi/2)=\pi/2.
\end{equation}

Then, the function  $m(r):=a\sin h(r),$ where $a:=1/{h'(0)},$
gives a Riemannian metric $dr^2+m(r)^2d\theta^2$ of a 2-sphere of revolution on the unit sphere ${\Sph^2}. $ 
This surface has a reflective symmetry with respect to the equator $r=\pi/2,$
since  $m(\pi-r)=m(r) $ holds on $[0,\pi] $ and  $m$ is extendable to a smooth odd function.

\begin{proposition}\label{prop2.5}
If $h''(x)> 0$ (respectively $h''(x)\geq 0)$ on $ (0,\pi/2),$ then the cut locus of a point $q\in r^{-1}(0,\pi)$ of the 2-sphere of revolution $(\Sph^2,dr^2+m(r)^2d\theta^2)$, where $m(r)=a\sin h(r),a=1/h'(0)$ is   a subarc (respectively a single point or a subarc) of the antipodal parallel $r=\pi-r(q).$
\end{proposition}
\begin{proof}
It is clear that 
$\sqrt{a^2-m^2(r)}/m'(r)=1/h'(r)$ on $(0,\pi/2).$  
Hence, by Lemma \ref{lem2.1}, the half period function
$\varphi_m(\nu)$ with respect to $m$ is strictly decreasing (respectively decreasing) on $(0,a).$
From \cite [Lemma 3.3, Theorem 3.5]{BCST},
our conclusion is clear.
$\qedd$
\end{proof}

It is well-known that 
the Gaussian curvature $G(q)$ at  a point $q\in r^{-1}(0,\pi)$ of the sphere $(\Sph^2,dr^2+m(r)^2d\theta^2)$ equals $(-m''/m)(r(q)).$  
Since $m'(r)=a\cos h(r)\cdot h'(r)$ and $m''(r)=-m(r)\cdot h'(r)^2+
a\cos h(r)\cdot h''(r),$ we get
$$G(q)=(-m''/m)(r(q))=(h'(r(q)))^2-\cot h(r(q))\cdot h''(r(q)).$$
If we define a function $\tilde G$ defined on $[0,\pi]$ by
$\tilde G(r(q)):=G(q),$
we get
\begin{equation}\label{2.9}
\tilde G(x)=h'(x)^2-\cot h(x)\cdot h''(x)
\end{equation}
on $[0,\pi].$

Choose a smooth odd function $h_0(x)$ satisfying the properties 
\eqref{2.6}, \eqref{2.7}, and
\begin{equation}\label{2.10}
h_0''(x)>0\quad {\text on}  \quad\:\: (0,\pi/2).
\end{equation}
By Proposition \ref{prop2.5}, the cut locus of  a point distinct from the poles on the 2-sphere of revolution
$(\Sph^2,dr^2+m_0(r)^2d\theta^2),$ where $m_0(r)=a\sin h_0(r),a=1/h_0'(0),$ 
is a subset of the  antipodal parallel of the point. 

\begin{example}\label{ex2.6}
The odd function $h_0(x):=x-\alpha \sin 2x$ satisfies the properties 
\eqref{2.6}, \eqref{2.7}, and \eqref{2.10} for each constant $\alpha \in(0,1/2).$
\end{example}

Let $\{R_n(x)\}_n$ denote a family of smooth odd functions on $[0,\pi]$
such that for each positive integer $n$
\begin{equation}\label{2.11}
R_n{}'(0)=0,
\end{equation}
\begin{equation}\label{2.12}
R_n(\pi-x)=-R_n(x)\quad {\text on\quad } [0,\pi],
\end{equation}
\begin{equation}\label{2.13}
\limsup_{n\to \infty} \:\:\sup_{x\in(0,\pi/2)}|R_n'(x)|/h'_0(x)<1,
\end{equation}
and 
\begin{equation}\label{2.14}
\limsup_{n\to \infty} \:\:\sup_{x\in(0,\pi/2)}|R_n''(x)|/h''_0(x)<1.
\end{equation}

By making use of the family $\{R_n(x)\}$ of smooth odd functions,
we introduce  a family of 2-spheres of revolution $\{(\Sph^2,dr^2+m_n(r)^2d\theta^2)
\}_n,$ where
\begin{equation} \label{2.15}
m_n(r):=a\cdot{\sin h_n(r)},  {\quad} 
a=1/h_0{}'(0)=1/h_n{}'(0), \quad h_n(r):=h_0(r)+R_n(r).
\end{equation}

Since the function $h_0$ satisfies \eqref{2.6}, \eqref{2.7}, 
and \eqref{2.10},  it follows from \eqref{2.11},\dots,\eqref{2.14} that there exists a constant $N_0>0$ such that for each $n>N_0, $ the smooth odd function  $h_n$ also satisfies  \eqref{2.6}, \eqref{2.7}, and \eqref{2.10}. Hence, by Proposition \ref{prop2.5}, we obtain:
\begin{proposition}\label{prop2.7}
For each $n>N_0,$
the function $m_n(r)$ gives a smooth Riemannian metric of  a 2-sphere of revolution
on the unit 2-sphere. The cut locus of a point $q\in r^{-1}(0,\pi)$ of the 2-sphere 
 $(\Sph^2,dr^2+m_n(r)^2d\theta^2)$ of revolution is  a subarc of the antipodal parallel of $q.$
\end{proposition}

By the equation \eqref{2.9},
the Gaussian curvature $\widetilde G_n(r(q))$ at  a point $q\in r^{-1}(0,\pi)$ of the 2-sphere 
$(\Sph^2,dr^2+m_n(r)^2d\theta^2)$ of revolution equals 
\begin{equation}\label{2.16}
\widetilde G_n(r(q))=h_n'(r(q))^2-\cot h_n(r(q))\cdot h_n''(r(q)).
\end{equation}

\begin{lemma}\label{lem2.8}
For each positive integer  $n,$
$|\sin nx|\leq n|\sin x|$ holds for all $x.$
\end{lemma}
\begin{proof}
The proof is clear by induction. In fact,
suppose that $|\sin nx|\leq n|\sin x|$ holds for some positive integer
$n=n_0.$ From the sine addition formula, it follows that
$\sin(n_0+1)x= \sin n_0 x\cdot\cos x+\sin x\cdot\cos n_0x.$ 
Thus, by the triangle inequality,
we get
$|\sin(n_0+1)x|\leq |\sin n_0x|+|\sin x|\leq (n_0+1)|\sin x|.$

\end{proof}
$\qedd$

Here, 
we introduce 
 a family $\{R_n^B(x)\}_n$ of smooth odd functions defined by
$$R_n^B(x):=B(x)\sin 2n^2x/n^5,$$
where 
$B$ denotes a smooth even function on $[0,\pi]$
satisfying
\begin{equation}\label{2.17}
B(\pi-x)=B(x)  \quad {\text on} \quad [0,\pi],
\end{equation}
and
\begin{equation}\label{2.18}
B(0)=B(\pi/2)=0.
\end{equation}

\begin{lemma}\label{lem2.9}
$\sup_{x\in(0,\pi/2)} |B(x)|/\sin 2x<+\infty$ 
and 
$\sup_{x\in(0,\pi/2)}|B'(x)|/\sin 2x<+\infty.$
\end{lemma}
\begin{proof}
From the equation \eqref{2.18}, it is trivial that 
$\sup_{x\in(0,\pi/2)} |B(x)|/\sin 2x<+\infty.$
Since the function $B$ is even,
we have
\begin{equation}\label{2.19}
B'(0)=0,
\end{equation}
and by differentiating the equation \eqref{2.17},
we obtain
$B'(\pi-x)=-B'(x).$  Hence, by substituting $x=\pi/2,$
we obtain
\begin{equation}\label {2.20}
B'(\pi/2)=0.
\end{equation}
By \eqref{2.19} and \eqref{2.20}, it is clear that $\sup_{x\in(0,\pi/2)}|B'(x)|/\sin 2x<+\infty.$

$\qedd$
\end{proof}

\begin{lemma}\label{lem2.10}
There exists a number $N_0>0$ such that
for any $n>N_0,$ and any $x\in [0,\pi]$
$|h_n^B{}'(x)|\leq 2, |h_n^B{}''(x)|\leq 2,$
where
$h_n^B(x):=x-\alpha\sin 2x +R_n^B(x),$ and $\alpha\in(0,1/2)$ is a constant.
\end{lemma}
\begin{proof}

It is easy to check that
\begin{equation}\label{2.21}
(B(x)\cdot\sin 2n^2x)'=B'(x)\cdot\sin 2n^2 x+2n^2B(x)\cdot\cos2n^2 x
\end{equation}
and
\begin{equation}\label{2.22}
(B(x)\cdot\sin 2n^2x)''=(B''(x)-4n^4B(x))\cdot\sin 2n^2 x+4n^2B'(x)\cdot\cos 2n^2 x.
\end{equation}
Therefore, we get
\begin{equation}\label{2.23}
|R_n^B{}'(x)|\leq(|B'(x)|+2n^2|B(x)|)/n^5
\end{equation}
\begin{equation}\label{2.24}
|R_n^B{}''(x)|\leq(|B''(x)\cdot\sin 2n^2x|+4n^4|B(x)|+4n^2|B'(x)|)/n^5.
\end{equation}
From the equations above, it is clear that $\lim_{n\to \infty}\sup_{x\in[0,\pi]}|R_n^B{}'(x)|=\lim_{n\to \infty}\sup_{x\in[0,\pi]}|R_n^B{}''(x)|=0.$
Since $|(x-\alpha\sin 2x)'|=|1-2\alpha\cos 2x|\leq 1+2\alpha<2, $ 
and $|(x-\alpha\sin 2x)''|=|4\alpha\sin 2x|\leq 4\alpha<2,$
the existence of the number $N_0$ is established.
$\qedd$
\end{proof}


\begin{lemma}\label{lem2.11}
For each $n,$ the function $R_n^B(x)$
satisfies \eqref{2.11},...,\eqref{2.14} for $h_0(x)=x-\alpha\sin 2x,\alpha\in(0,1/2).$ 
\end{lemma}
\begin{proof}
From \eqref{2.17} and \eqref{2.18}, it is clear that the function $R_n^B$ satisfies
\eqref{2.11} and \eqref{2.12}.
Applying Lemma \ref{lem2.8} to the function $\sin 2n^2x,$
we obtain
$|\sin 2n^2x|\leq n^2|\sin 2x|.$
Therefore,  by  \eqref{2.24},
$$ |R_n^B{}''(x)|/\sin 2x\leq \left(n^2|B''(x)|+4n^4|B|/\sin 2x+4n^2|B'(x)|/\sin 2x\right)/n^5$$
hold for all $x\in(0,\pi/2).$
Since $h_0{}'(x)\geq 1-2\alpha>0,$ and $h_0''(x)=4\alpha\sin 2x,$  we get, by \eqref{2.23} and Lemma \ref{lem2.9},
$$\lim_{n\to \infty } \sup_{x\in(0,\pi/2)} |R_n^B{}'(x)|/h_0'(x)=0\quad {\rm and}\quad
 \lim_{n\to \infty } \sup_{x\in(0,\pi/2)} |R_n^B{}''(x)|/h_0''(x)=0.$$
These equations imply that the family  $\{R_n^B(x)\}_n$ satisfies 
the property \eqref{2.13} and \eqref{2.14} for $h_0(x)=x-\alpha\sin 2x.$
$\qedd$
\end{proof}
Hence, 
by Proposition \ref{prop2.7} and Lemma \ref{lem2.11}, we obtain:
\begin{proposition}\label{prop2.12}
For any sufficiently large $n$, the cut locus of a point $q\in r^{-1}(0,\pi)$ of the 
2-sphere of revolution $(\Sph^2,dr^2+m_n(r)^2d\theta^2),$ where $m_n(r)= (1-2\alpha)^{-1} \sin (r-\alpha\sin 2r+R_n^B(r)),$ is a  subarc of the antipodal parallel  of $q.$  
\end{proposition}


\section{Latitudinally  corrugated  2-spheres }
In this section,   Theorem A  will be proven more generally, so that we get the theorem as a corollary.
Let $\widetilde G_n(r(q))$ denote the Gaussian curvature at a point  $q\in r^{-1}(0,\pi/2)$ of the 2-sphere of revolution $(\Sph^2, dr^2+m_n(r)^2d\theta^2),$ where $m_n(r)=(1-2\alpha)^{-1}\sin\left(r-\alpha\sin 2r+R_n^B(r)\right), $
$R_n^B(r)=B(r)\sin 2n^2r/n^5,$ and $\alpha\in(0,1/2)$ is a constant.

\begin{lemma}\label{lem3.1}
For each $t\in(0,\pi/2),$
\begin{equation}\label{3.1}
2\widetilde G_n{}'(x)\cdot\sin^2h_n^B(x)=2(2-\cos 2h_n^B(x))h_n^B{}'(x)\cdot h_n^B{}''(x)-(\sin 2h_n^B(x))\cdot h_n^B{}'''(x)
\end{equation}
holds,
where $h_n^B(x):=x-\alpha\sin 2x +R_n^B(x).$ 
\end{lemma}
\begin{proof}
By \eqref{2.16}, we have
$$\widetilde G_n(x)=h'(x)^2-\cot h(x)\cdot h''(x).$$
Here we set $h(x):=h_n^B(x)$ for simplicity. From a direct computation,
it follows that 
$$\widetilde G_n{}'(x)=h'(x)h''(x)( 2+1/\sin^2h(x))-\cot h(x)\cdot h'''(x).$$
Hence, we obtain
$$2\widetilde G_n{}'(x)\cdot\sin^2h(x)=2(2-\cos 2h(x))h'(x)h''(x)-\sin 2h(x)\cdot h'''(x).$$
$\qedd$
\end{proof}

\begin{lemma}\label{lem3.2}
For each number $\delta\in(0,\pi/3),$
there exist  positive numbers $c(\delta)$ and  $N_0(\delta,B),$ depending only on $\delta,$
and depending only on $\delta$ and the maximum of function $|B(x)|$ on $[0,\pi/2],$ respectively,  such that
$$\sin 2h_n^B(x)\geq c(\delta) $$
for all $n>N_0(\delta,B)$ and all $x\in[\delta,(\pi-\delta)/2].$
\end{lemma}

\begin{proof}
Since $h_0(x)=x-\alpha\sin 2x$ is increasing,
$$0<\delta-\alpha\sin 2\delta\leq h_0(x)\leq (\pi-\delta)/2-\alpha\sin\delta$$
for all $x\in[\delta,(\pi-\delta)/2].$
Since the positive number $\alpha$ is less than $1/2,$
we obtain 
$\delta-\alpha\sin 2\delta>(2\delta-\sin 2\delta)/2>0,$ and $(\pi-\delta)/2-\alpha\sin\delta<(\pi-\delta)/2.$
Hence, for all $x\in[\delta,(\pi-\delta)/2],$
\begin{equation}\label{3.2}
0<(2\delta-\sin 2\delta)/2\leq h_0(x)\leq (\pi-\delta)/2.
\end{equation}
We can choose a number $N_0(\delta,B)$ so as to satisfy
$\sup_{x\in[0,\pi/2]}|R_n^B(x)|<\epsilon_0(\delta), $ where
$\epsilon_0(\delta):=(2\delta-\sin2\delta)/8,$ 
  for all $n>N_0(\delta,B),$ since $R_N^B(x)=B(x)\sin 2n^2x/n^5.$
By the triangle inequality,
$$h_0(x)-|R_n^B(x)|\leq h_n^B(x)\leq h_0(x)+|R_n^B(x)|.$$
By \eqref{3.2},
for all 
$n>N_0(\delta,B)$ and all $x\in[\delta,(\pi-\delta)/2],$
$$ 6\epsilon_0(\delta)\leq 2h_n^B(x)\leq\pi-\delta+2\epsilon_0(\delta), $$
and hence
$$\sin 2h_n^B(x)\geq\min( \sin 6\epsilon_0(\delta), \sin(\delta-2\epsilon_0(\delta)) ).$$
This implies  
that the positive constant $c(\delta)$ is $\min( \sin 6\epsilon_0(\delta), \sin(\delta-2\epsilon_0(\delta)) ).$ 

$\qedd$
\end{proof}
\begin{lemma}\label{lem3.3}
For each positive integer $k\leq n^2,$
$$h_n^B{}'''(t_k^{(n)})=8\alpha\cos 2t_k^{(n)}+(-1)^k\left(6n^{-3}B''(t_k^{(n)})-8nB(t_k^{(n)})\right),$$
where
$t_k^{(n)}=k\pi/(2n^2)\in[0,\pi/2].$
\end{lemma}
\begin{proof}
Since 
$\sin2n^2t_k^{(n)}=0,\cos 2n^2t_k^{(n)}=(-1)^k,$
it follows from the equation
\eqref{2.22} that 
$$(B(x)\sin 2n^2 x)'''\bigr|_{x=t_k^{(n)}}=2n^2(-1)^k\left(B''(t_k^{(n)})-4n^4B(t_k^{(n)})\right )+4n^2(-1)^kB''(t_k^{(n)})$$
holds.
Since $h_n^B{}'''(x)=8\alpha\cos 2x+(B(x)\sin 2n^2x)'''/n^5,$
we get
$h_n^B{}'''(t_k^{(n)})=8\alpha\cos 2t_k^{(n)}+(-1)^k\left(6n^{-3}B''(t_k^{(n)})-8nB(t_k^{(n)})\right).$

$\qedd$
\end{proof}
\begin{theorem}\label{th3.4}
Let $B(x)$ denote a smooth even function satisfying the properties   \eqref{2.17}, and \eqref{2.18}. Then, for any sufficiently large $n,$ the odd function 
\begin{equation}\label{3.3}
m_n(r):=(1-2\alpha)^{-1}\sin(r-\alpha\sin 2r+B(r)\sin 2n^2r/n^5 ),
\end{equation}
 where $\alpha\in(0,1/2)$ is a constant, gives a Riemannian metric $dr^2+m_n(r)^2d\theta^2$ of a 2-sphere of revolution  on the unit sphere $\Sph^2,$ and the cut locus  of each point 
$p\in r^{-1}(0,\pi)$ of the 2-sphere of revolution $(\Sph^2,dr^2+m_n(r)^2d\theta^2),$
is a subarc of the antipodal parallel $r=\pi-r(p)$ for all sufficiently large $n.$
Furthermore, if the function $B(x)$ is not identically zero on the open interval  $(0,\pi/2),$ then
the number of points on each meridian which assume a local maximum or minimum of the Gaussian curvature function on the meridian goes to infinity as $n$ tends to infinity.
\end{theorem}
\begin{proof}
Since the first claim is clear from Propositions \ref{prop2.7} and \ref{prop2.12},  we omit the proof.
By Lemmas \ref{lem3.1} and \ref{lem3.3},
\begin{equation}\label{3.4}
2\widetilde G_n{}'(t_k^{(n)})\cdot\sin^2h_n^B(t_k^{(n)})=f_{n,k,B}+(-1)^k8n\cdot\sin 2h_n^B(t_k^{(n)})\cdot B(t_k^{(n)}),
\end{equation}
where 
\begin{multline*}
f_{n,k,B}:=2(2-\cos2h_n^B(t_k^{(n)}))\cdot h_n^B{}'(t_k^{(n)})\cdot h_n^B{}''(t_k^{(n)})\\
-\sin2h_n^B(t_k^{(n)})
(8\alpha\cos2t_k^{(n)}+(-1)^k6n^{-3}B''(t_k^{(n)})).
\end{multline*}
From lemma \ref{lem2.10} and the triangle inequality, it follows that
for any $n>N_0,$
\begin{equation}\label{3.5}
|f_{n,k,B}|<3\cdot2^3+8\alpha+6N_0^{-3}\max_{[0,\pi/2]}|B''(x)|<28+6N_0^{-3}\max_{[0,\pi/2]}|B''(x)|.
\end{equation}
Since $B(x)$ is not identically zero, we can find a closed interval $I\subset (0,\pi/2)$ such that $\min_{I}|B(x)|>0.$

Choose any number $\delta\in(0,\pi/3)$  satisfying $I\subset(\delta, (\pi-\delta)/2)$ 
  and fix it.
By Lemma \ref{lem3.2}, $\sin 2h_n^B(x)\geq c(\delta)$ for all $x\in I$ and all $n>N_0(\delta,B).$
Choose any integer $N_1>N_0$ so as to satisfy 
$$28+6N_0^{-3}\max_{[0,\pi/2]}|B''(x)|<4N_1c(\delta)\min_{ I}|B(x)|.$$
For each $n>\max\{N_1,N_0(\delta,B)\},$ let $I_n$ denote the set of all positive integers $k$ satisfying $t_k^{(n)}=k\pi/(2n^2)\in I.$
Then, by \eqref{3.5}, for any $n>\max\{N_1, N_0(\delta,B)\},$ and any positive integer $k$ satisfying $k\in I_n,$
$|\epsilon_{n,k,B}|<1/2,$
where $\epsilon_{n,k,B}:=f_{n,k,B}\left ( (-1)^k8n\cdot\sin 2h_n^B( t_k^{(n)})\cdot B(t_k^{(n)}) \right)^{-1}.$
Thus, from \eqref{3.4}, we get
\begin{equation}\label{3.6}
2\widetilde G_n{}'(t_k^{(n)})\cdot\sin^2h_n^B(t_k^{(n)})=(-1)^k8n\cdot\sin 2h_n^B(t_k^{(n)})\cdot B(t_k^{(n)})(\epsilon_{n,k,B}+1).
\end{equation}
Note that $\epsilon_{n,k,B}+1>0$ for all $k\in I_n,n>\max\{N_1,N(\delta,B)\}.$
Hence,
$\widetilde G_n{}'(t_{k}^{(n)})\cdot\widetilde G_n{}'(t_{k+1}^{(n)})<0$
for all integers $k\in I_n$ with $k+1\in I_n.$
Therefore, the claim of our theorem is  clear, since the number of the elements of the set $I_n$ tends to infinity as the number $n$ goes to infinity.

$\qedd$
\end{proof}

\noindent
{\it Proof of Theorem A}\medskip\\
The even function $B(x)=\sin^2 2x$ satisfies the properties \eqref{2.17}, and \eqref{2.18}.
By substituting $\alpha=1/3\in(0,1/2),$ and $B(x)=\sin^2 2x$ in the equation \eqref{3.3},
we get  the equation
 \eqref{1.5} and hence  Theorem A.

\begin{example}\label{ex3.5}
For any smooth function $f,$ the smooth even function
$B(x)=\sin^22x\cdot f(\cos^2 x)$ satisfies \eqref{2.17} and \eqref{2.18}.
\end{example}
\section{Remarks and an open problem} 

It was proven in \cite[Theorem 3.1, Lemma 3.4]{ST2} that the cut locus of a point $p\in \theta^{-1}(0)$ on a 2-sphere of revolution $(\Sph^2,dr^2+m(r)^2d\theta^2)$
(with reflective symmetry with respect to the equator)  is a subarc of $\theta^{-1}(\pi)$, and that the half period function is increasing, if the Gaussian curvature function is decreasing from the south pole to the point on the equator along each meridian.
We do not know whether the increasing property of the half period function  is a sufficient condition for a 2-sphere of revolution to have the same cut locus structure as above.
A sufficient condition for a 2-sphere of revolution to have the same cut locus structure is given in \cite[Theorem 3.7]{BCST}.
The following open problem is a dual of {\sc Theorem A}.

\begin{problem}
Find a family of 2-spheres of revolution $\{M_n\}_n$ such that the cut locus of
each point of $\theta^{-1}(0)$ is a subarc of the opposite half meridian $\theta^{-1}(\pi)$ and the number 
of points on each meridian which assume a local maximum or minimum of the Gaussian curvature function on the meridian goes to infinity as $n$ tends to infinity.
\end{problem}


\bigskip

\noindent School of Science,
Department of Mathematics,\\
Tokai University,
Hiratsuka City, Kanagawa Pref., 259\,--\,1292,
Japan

\medskip
\noindent
Toyohiro Akamatsu
\noindent
{\tt akamatsu@tokai-u.jp }

\medskip
\noindent
Minoru Tanaka
\noindent
{\tt
tanaka@tokai-u.jp}

\medskip
\noindent
Masaru Yamaguchi
\medskip
\noindent
{ \tt ym24896@tsc.u-tokai.ac.jp}
\medskip
\noindent

\noindent Faculty of Economics,\\
Hosei University,
4342 Aihara-machi,
Machida, Tokyo, 194\,--\,0298,
Japan

and

\noindent Molecular and Integrative Biosciences Research Programme,\\
Faculty of Biological and Environmental Sciences,\\
University of Helsinki, Finland

\medskip
\noindent
Robert Sinclair
\medskip
\noindent
{\tt  sinclair.robert.28@hosei.ac.jp}

\end{document}